\documentclass[a4paper,12pt,reqno]{amsart}
\usepackage{amsfonts,amssymb,amscd,amsmath,latexsym,amsbsy,enumerate,a4wide}

\newtheorem{thm}{Theorem}[section]

\newtheorem{lem}[thm]{Lemma}

\newtheorem{prop}[thm]{Proposition}

\newtheorem{conj}[thm]{Conjecture}

\theoremstyle{definition}
\newtheorem{rem}[thm]{Remark}

\numberwithin{equation}{section}
\def\al{\alpha}
\def\be{\beta}
\def\ga{\gamma}
\def\de{\delta}

\def\Q{\mathbb{Q}}

\def\C{\mathbb{C}}
\def\N{\mathbb{N}}

\def\sgn{\text{\rm sgn}}

\newcommand{\SU}{\mathrm{SU}}

\title[Mathieu Conjecture]
{On the Mathieu conjecture for $SU(2)$}

\author{Teun Dings, Erik Koelink}
\address{Radboud Universiteit, IMAPP, 
Heyendaalseweg 135, 
6525 GL Nijmegen, 
The Netherlands}
\email{teundings@gmail.com, e.koelink@math.ru.nl}
\date{\today}

\begin{document}

\begin{abstract} We study the Mathieu Conjecture for $SU(2)$ using the matrix elements of its unitary irreducible representations. 
We state a conjecture for the particular case $SU(2)$ implying the Mathieu Conjecture for $SU(2)$. 
\end{abstract}

\maketitle
\section{Introduction}\label{sec:intro}

\begin{conj}[Mathieu \cite{Math}] \label{conj:mathieu}
Let $G$ be a compact connected Lie group and let $f$ be complex-valued $G$-finite function on $G$ such that 
$\int_G f^P(g)\, dg=0$ for every $P\in \mathbb{N}_{>0}$. Then 
for any complex-valued $G$-finite function $h$ on $G$ we have $\int_G f^P(g)h(g)\, dg=0$ for $P \gg 0$.
\end{conj}

The Mathieu Conjecture \ref{conj:mathieu} dates back to 1997 and  is closely related to the Jacobian conjecture, 
since it actually implies the Jacobian
conjecture, see \cite{Math}. See van den Essen \cite{Esse}, Smale \cite{Smal} for more 
information on the history of the  Jacobian conjecture. 
The Mathieu Conjecture \ref{conj:mathieu} was proved for abelian compact groups by Duistermaat and Van der Kallen \cite{DuisvdK} in 1998. 
We study the Mathieu Conjecture \ref{conj:mathieu} for the 
case $G=SU(2)$. Using explicit formulas for the Haar measure and known representation theoretic properties of $SU(2)$ we make 
the Mathieu Conjecture \ref{conj:mathieu} more explicit. In particular, we use the fact that
$SU(2)$-finite functions are finite linear combinations of matrix elements of finite dimensional
irreducible representations of $SU(2)$ and that the matrix elements behave well under a subgroup $K\cong U(1)$ according to 
suitable characters. 
Note that the Mathieu Conjecture  \ref{conj:mathieu} is linear in the $G$-finite function $h$, but not in the $G$-finite function $f$. 
By the Peter-Weyl theorem, any $SU(2)$-finite function is the finite linear combination of 
matrix elements of irreducible representations.
After recalling the necessary results on $SU(2)$ in Section \ref{sec:SU2}, we show in Section \ref{sec:MathieuConjforSU2small} 
the validity of the Mathieu Conjecture \ref{conj:mathieu} for $f$ a single matrix element
or a sum of two matrix elements. For the sum of three matrix elements there is a partial result.  
These considerations lead to Conjecture \ref{conj:nonzerointegralpowers}, and Theorem \ref{thm:alternativMC} shows that this conjecture implies
the Mathieu Conjecture \ref{conj:mathieu} for $SU(2)$. 
Conjecture \ref{conj:nonzerointegralpowers} describes the condition $\int_{SU(2)} f(g)^P\, dg=0$ for all $P>0$
in terms of a support condition on the characters of the abelian subgroup $U(1)$ of $SU(2)$ acting from the 
left and right on the 
individual matrix elements occurring in $f$. 

We note that the Mathieu Conjecture \ref{conj:mathieu} for bi-$K$-invariant functions is settled by Francoise et al. \cite[Cor.~4.1]{FranPYZ},
since the bi-$K$-invariant $SU(2)$-finite functions are the polynomials on $[-1,1]$. 

\section{$SU(2)$}\label{sec:SU2}

We briefly recall some required notions of $SU(2)$. Details can be found
in e.g. \cite{Vile}, \cite{VileK}.
Let $k(\phi) = \begin{pmatrix}
e^{\frac{i}{2}\phi} & 0 \\
0 & e^{-\frac{i}{2}\phi}
\end{pmatrix}$ and $a(\theta) =\begin{pmatrix}
\cos \frac{\theta}{2} & i \sin \frac{\theta}{2} \\
i \sin \frac{\theta}{2} & \cos \frac{\theta}{2}
\end{pmatrix}$ be elements of $SU(2)$, then any element $g\in\SU(2)$ can be expressed in terms
of Euler angles $g=k(\phi)a(\theta)k(\psi)$ with $\phi\in[0,2\pi)$, 
$\theta\in (0,\pi)$, $\psi \in [-2\pi,2\pi)$. In terms of the Euler angles the Haar integral is,
cf \cite[III, \S 6.1, (5)]{Vile},
\begin{equation}\label{eq:HaaronSU2}
\int_{SU(2)}f(g)\, dg=\frac{1}{16\pi^2}\int_0^{2\pi}\int_0^{\pi}\int_{-2\pi}^{2\pi} F(\phi,\theta,\psi) \sin\theta\,d\psi\,d\theta\, d\phi,
\end{equation}
where $F(\phi,\theta,\psi) =f\bigl(k(\phi)a(\theta)k(\psi)\bigr)$.
Denote the subgroup $K\cong U(1)$ generated by $k(\phi)$. For a function $f$ transforming by a non-trivial $K$-character 
under left- or right multiplication by $K$, we have  $\int_{SU(2)}f(g)\, dg=0$ by \eqref{eq:HaaronSU2}. 
The subgroup generated by $a(\theta)$ is the group $SO(2)$. 

The finite-dimensional irreducible representations are labeled by the spin $\ell\in \frac12\N$ and are of dimension $2\ell+1$. 
The standard basis for the representation space is labeled as $\{-\ell,-\ell+1, \dots, \ell\}$, and the corresponding matrix elements 
$t^\ell_{m,n}$ are $SU(2)$-finite functions, and any $SU(2)$-finite function is a finite linear combination of matrix elements of irreducible finite-dimensional representations. The matrix-elements $t^\ell_{m,n}$ 
behave well according to left and right action by $K$, cf. \cite[III, \S 3.3, (3)]{Vile}
\begin{equation}\label{eq:leftrighKbehaviourtlmn}
t^\ell_{m,n}(k(\phi)g) = e^{-im\phi} \, t^\ell_{m,n}(g), \qquad
t^\ell_{m,n}(gk(\psi)) = e^{-in\psi} \, t^\ell_{m,n}(g).
\end{equation}
In particular, $t^0_{0,0}(g)=1$, and the algebra of bi-$K$-invariant $SU(2)$-finite functions consists of 
finite linear combinations of $t^\ell_{0,0}$, $\ell\in\N$. 
For $\ell\in \N$ we have $t^\ell_{0,0}(a(\theta))= P_\ell(\cos\theta)$, cf. \cite[III, \S 3.9, (5)]{Vile}
where $P_\ell$
is the Legendre polynomial in its standard normalisation $P_\ell(1)=1$, \cite[\S 4.5]{Isma}, \cite[\S 1.8.3]{KoekS}, 
which is real-valued on $[-1,1]$. The Legendre polynomials are orthogonal on $[-1,1]$ with respect to the uniform measure;
$\int_{-1}^1 P_n(x)P_m(x) \, dx = \de_{n,m} 2/(2n+1)$. 
Moreover, the Schur orthogonality relations are, \cite[III, \S 6.2, (1)]{Vile}
\begin{equation} \label{eq:Schurortho}
\int_{SU(2)}t^{\ell_1}_{m,n}(g)\overline{t^{\ell_2}_{p,q}(g)}\,dg=\frac{1}{2\ell_1+1}
\de_{\ell_1,\ell_2}\de_{m,p}\de_{n,q},
\end{equation}
which in case $m=n=p=q=0$ give the orthogonality for the Legendre polynomials. 

\section{The Mathieu Conjecture for $SU(2)$ for simple $f$}\label{sec:MathieuConjforSU2small}

We start using some simple observations related to the condition in the 
Mathieu Conjecture \ref{conj:mathieu} for $G=SU(2)$. 
Firstly, by the Schur orthogonality relations \eqref{eq:Schurortho}
\begin{equation}\label{eq:integraltlnmis0}
\int_{SU(2)}t^\ell_{m,n}(g)\, dg\neq 0 \iff \ell=0.
\end{equation}
Secondly, by the left and right $K$-behaviour of the matrix elements \eqref{eq:leftrighKbehaviourtlmn}
and the Haar measure in Euler angles \eqref{eq:HaaronSU2} we see 
\begin{equation}\label{eq:integralproducttlnmisnon0}
\int_{SU(2)} \left(t^{\ell_1}_{m_1,n_1}\right)^{\al_1}(g) \cdots \left(t^{\ell_k}_{m_k,n_k}\right)^{\al_k}(g) 
\, dg \neq 0 \quad  \Longrightarrow\quad \sum_{i=1}^{k} \al_i m_i=0=\sum_{i=1}^{k} \al_in_i 
\end{equation}
for $\al_i \in \N$, 
$\ell_i \in \frac{1}{2}\N$ and $m_i,n_i \in \{-\ell_i,  \dots, \ell_i\}$. 

\begin{lem}\label{lem:1matrixeltpowerP}
$\int_{SU(2)} \left(t^\ell_{m,n}\right)^P(g)\, dg=0$ for all integer $P>0$ if and only if
$m \neq 0$ or $n \neq 0$.
\end{lem}

\begin{proof} The implication $\Leftarrow$ follows from \eqref{eq:integralproducttlnmisnon0}. To prove
the other implication, we observe that for $\ell\in\N$
\[
\int_{SU(2)} \left(t^\ell_{0,0}(g)\right)^2\, dg = \frac12 \int_0^\pi \left( P_\ell(\cos\theta)\right)^2\, \sin\theta \, 
d\theta = \frac12 \int_{-1}^1 \left( P_\ell(x)\right)^2\, 
dx > 0. 
\qedhere
\]
\end{proof}

Now we can verify the Mathieu Conjecture \ref{conj:mathieu} in the case $f$ consists of one 
matrix element.

\begin{prop}\label{prop:MCtrueforsinglematrixelt} 
The Mathieu Conjecture \ref{conj:mathieu} is true for $G=SU(2)$ with $f$  a single matrix element 
$f=t^{\ell}_{m,n}$. 
\end{prop}

\begin{proof} Since all non-negative powers of $f$ integrate to zero, Lemma \ref{lem:1matrixeltpowerP} shows that
$m\not=0$ or $n\not=0$, so in particular $\ell\not=0$. Let $h=t^{\ell_0}_{a,b}$. 
We assume $m\not=0$, the case $n\not=0$ being similar. 
By \eqref{eq:integralproducttlnmisnon0} we see that $Pm+a\not=0$ implies 
$\int_{SU(2)}\bigl(f(g)\bigr)^P h(g)\, dg=0$, which is the case for $P> |a|/|m|$. 
\end{proof}

The same strategy can also be employed to deal with $f= A_1 t^{\ell_1}_{m_1,n_1}+ A_2 t^{\ell_2}_{m_2,n_2}$, where
$A_i\in\C$, assuming $A_1\not=0\not=A_2$ and $(\ell_1,m_1,n_1)\not=(\ell_2,m_2,n_2)$. Note 
\begin{equation}\label{eq:poweroff=2sum}
\int_{SU(2)}\bigl(f(g)\bigr)^P\, dg = 
\sum_{\al=0}^P \binom{P}{\al} A_1^\al A_2^{P-\al} \int_{SU(2)} \bigl(t^{\ell_1}_{m_1,n_1}\bigr)^{\al}(g) 
\bigl(t^{\ell_2}_{m_2,n_2}\bigr)^{P-\al}(g)\, dg.
\end{equation}

\begin{lem}\label{lem:MCfortwomatrixelt}
Let $f$ be as above with at least one of $(m_1,m_2,n_1,n_2)$ non-zero, then
\[
\exists P>0:\quad \int_{SU(2)}\bigl(f(g)\bigr)^P\, dg\not= 0 \quad\Longleftrightarrow \quad 
\det\begin{pmatrix} m_1& m_2 \\n_1 & n_2\end{pmatrix} = 0 
\wedge m_1m_2 \leq 0 \wedge n_1n_2 \leq 0.
\]
\end{lem}

\begin{rem}\label{rmk:supportfor2matrixelt} Note that the condition in Lemma \ref{lem:MCfortwomatrixelt} means 
that $(0,0)$ is on the line segment from $(m_1,n_1)$ to $(m_2,n_2)$. 
\end{rem}

\begin{proof}
$\Rightarrow$: Since at least one term in the right hand side of \eqref{eq:poweroff=2sum} has to be non-zero, \eqref{eq:integralproducttlnmisnon0} 
shows that $m_1\alpha+m_2(P-\al)=0=n_1\alpha+n_2(P-\al)$, which gives the result.

$\Leftarrow$: Note that $\dim \text{Ker} \begin{pmatrix} m_1& m_2 \\n_1 & n_2\end{pmatrix} =1$. 
Pick a solution $(\al,\be)\in \N^2$ to $m_1\al+m_2\be=0=n_1\al+n_2\be$, and put $M=\al+\be$. 
Then 
\begin{equation}\label{eq:powerMintegralsingleterm}
\int_{SU(2)}\bigl(f(g)\bigr)^M\, dg \, = \, \binom{\al+\be}{\al} A_1^\al A_2^\be 
\int_{SU(2)} \bigl(t^{\ell_1}_{m_1,n_1}\bigr)^{\al}(g) \bigl(t^{\ell_2}_{m_2,n_2}\bigr)^{\be}(g)\, dg,
\end{equation}
using \eqref{eq:integralproducttlnmisnon0}, since for $\ga\not=0$
\[
\begin{pmatrix} m_1& m_2 \\n_1 & n_2\end{pmatrix} \begin{pmatrix} \al+\ga \\ \be-\ga\end{pmatrix} 
=  \ga \begin{pmatrix} m_1& m_2 \\n_1 & n_2\end{pmatrix} \begin{pmatrix} 1 \\ -1 \end{pmatrix}
\not= \begin{pmatrix} 0 \\ 0\end{pmatrix}
\]
since the kernel is one-dimensional. 
The integrand on the right hand side of \eqref{eq:powerMintegralsingleterm} is a bi-$K$-invariant function, 
so that by \eqref{eq:HaaronSU2} we can restrict to the integral over $g=a(\theta)$, $\theta\in [0,\pi]$. 
By \cite[III, \S 3,(3),(4)]{Vile} the integrand in $a(\theta)$ is real-valued.
In case the integral is non-zero we are done. 
Otherwise, we put $P=2M$, and then in the same way there is again at most one non-zero integral in 
the right hand side of \eqref{eq:poweroff=2sum}, namely for $(2\al,2\be)$. The integral can be restricted to $SO(2)$
as before. 
Since this is the integral of a square, since the function 
$\bigl(t^{\ell_1}_{m_1,n_1}\bigr)^{\al}(a(\theta)) \bigl(t^{\ell_2}_{m_2,n_2}\bigr)^{\be}(a(\theta))$ is real, 
the integral is non-zero. 
\end{proof}

\begin{prop}\label{prop:MCtruefortwomatrixelt} 
The Mathieu Conjecture \ref{conj:mathieu} is true for $G=SU(2)$ with $f$ a sum of two matrix element 
$f= A_1 t^{\ell_1}_{m_1,n_1}+ A_2 t^{\ell_2}_{m_2,n_2}$, where
$A_1\not=0\not=A_2$ and $(\ell_1,m_1,n_1)\not=(\ell_2,m_2,n_2)$. 
\end{prop}

\begin{proof} It suffices to take $h=t^{\ell}_{a,b}$ and to assume that $\int_{SU(2)} (f(g))^P\, dg=0$ for all $P>0$. 
We need to show that $\int_{SU(2)} (f(g))^P t^\ell_{a,b}(g)\, dg$ vanishes for sufficiently large $P$.

First assume that not all of $m_i$'s and $n_i$'s are zero, then by Lemma \ref{lem:MCfortwomatrixelt}
we have $m_1m_2>0$ or $n_1n_2>0$ or $\det\begin{pmatrix} m_1& m_2 \\n_1 & n_2\end{pmatrix} \not= 0$. 
Consider the last
case, then by \eqref{eq:integralproducttlnmisnon0}, \eqref{eq:poweroff=2sum} 
we see that $\int_{SU(2)} (f(g))^P t^\ell_{a,b}(g)\, dg$ 
can only be non-zero if 
\[
m_1\al+ m_2\be = -a, \quad n_1\al+n_2\be =-b, \qquad \al+\be=P, \  \al,\be \in\N.
\]
The first two equations have a unique solution $(\al_0,\be_0)\in \Q^2$. In case 
$(\al_0,\be_0)\in\N^2$, we see that for all $P>\al_0+\be_0$ the integral is zero.
In case $m_1m_2>0$, we consider $m_1\al+m_2\be+a=0$. In case $\sgn(m_1)=\sgn(a)$, we have no
solution $(\al,\be)\in \N^2$, 
so that integral is zero using \eqref{eq:integralproducttlnmisnon0}, \eqref{eq:poweroff=2sum}. 
In case $\sgn(m_1)=-\sgn(a)$, we see that
the integral is zero for $P> |a|/\min(|m_1|, |m_2|)$.  The case $n_1n_2>0$ is dealt with analogously.

In case $m_1=m_2=n_1=n_2=0$, $f$ is a bi-$K$-invariant function, and 
\[
\int_{SU(2)} (f(g))^P\, dg = \frac12 \int_{-1}^1 \bigl( A_1P_{\ell_1}(x) + A_2P_{\ell_2}(x)\bigr)^P\, dx. 
\]
By Boyarchenko's result, see \cite[Cor.~4.1]{FranPYZ}, there is no polynomial $f$ 
such that $\int_{-1}^1 (f(x))^P\, dx=0$ for all $P>0$, so the 
Mathieu Conjecture \ref{conj:mathieu} is trivially valid in this case.
\end{proof}

The fact that at most one term in the binomial expansion leads to a non-zero
integral is typical for $f$ a linear combination of two matrix elements. For a combination of
three matrix-elements it gets more complicated.

\begin{prop}\label{prop:MCfor3matrixelements}
Let $f= \sum_{i=1}^3 A_i t^{\ell_i}_{m_i,n_i}$ with $A_i\not=0$ for all $i$ and let $(\ell_i,m_i,n_i)$ be mutually different.
Assume that $M= \begin{pmatrix}
1 & 1 & 1 \\
m_1 & m_2 & m_3 \\
n_1 & n_2 & n_3
\end{pmatrix}$ has $\text{\rm rank}(M)\not=2$. Then the Mathieu Conjecture \ref{conj:mathieu} is valid for $f$.
\end{prop}

\begin{proof} The analogue of \eqref{eq:poweroff=2sum} is the trinomial expansion 
\begin{equation}\label{eq:poweroff=3sum}
\int_{SU(2)} f^{P}(g) \,\mathrm{d}g = \sum_{\stackrel{\scriptstyle{\alpha_1+\alpha_2+\alpha_3=P}}{\scriptstyle{\al_i\in \N}}} 
\binom{P}{\alpha_1,\alpha_2,\alpha_3} \prod_{i=1}^3 A_i^{\alpha_i} 
\int_{SU(2)} \prod_{i=1}^3 \left(t_{m_in_i}^{l_i}\right)^{\alpha_i}(g) \, dg.
\end{equation}
As before, it suffices to consider the case $h=t^\ell_{a,b}$. 
We have to consider the cases $\text{\rm rank}(M)=1$ and $\text{\rm rank}(M)=3$. In the first case 
$m_i=m$ and $n_i=n$ for all $i$, and the integral in \eqref{eq:poweroff=3sum} is zero if $m\not=0$ or $n\not=0$ 
by \eqref{eq:integralproducttlnmisnon0}. In case $m\not=0$ we see that 
\begin{equation}\label{eq:poweroff=3sumh}
\int_{SU(2)} f^{P}(g) t^\ell_{a,b}(g) \,\mathrm{d}g = \sum_{\stackrel{\scriptstyle{\alpha_1+\alpha_2+\alpha_3=P}}{\scriptstyle{\al_i\in \N}}} 
\binom{P}{\alpha_1,\alpha_2,\alpha_3} \prod_{i=1}^3 A_i^{\alpha_i} 
\int_{SU(2)} \prod_{i=1}^3 \left(t_{m_in_i}^{l_i}\right)^{\alpha_i}(g)t^\ell_{a,b}(g) \, dg
\end{equation}
can only be non-zero if $Pm+a=0$, so that for $P>|a|/|m|$ the integral is zero. The case $n\not=0$ is
analogous. In case $m=n=0$, we see that the condition in the 
Mathieu Conjecture is not valid using \cite[Cor.~4.1]{FranPYZ} as in the proof of Proposition \ref{prop:MCtruefortwomatrixelt}.

In case $\text{\rm rank}(M)=3$, $M$ is invertible with $M^{-1}$ having rational entries. In particular, for each $P\in\N$ there is at 
most one term in the right hand side of \eqref{eq:poweroff=3sum} which can be non-zero, namely for 
$\overrightarrow{\al}_P=\begin{pmatrix} \al_1\\\al_2\\ \al_3\end{pmatrix} = M^{-1} \begin{pmatrix} P\\ 0 \\0\end{pmatrix}$ under the additional condition  
$\overrightarrow{\al}_P\in \N^3$. Assuming that this is the case, we see that, analogous to the proof of Proposition \ref{prop:MCtruefortwomatrixelt},
$\int_{SU(2)} f^{2P}(g) \,\mathrm{d}g \not=0$. 

So we need to consider the case that $\overrightarrow{\al}_P\not\in \N^3$ for all $P>0$. Then the integral in \eqref{eq:poweroff=3sumh} can only be 
non-zero in case  
\begin{gather*}
M^{-1}\begin{pmatrix} P \\ -a \\ -b \end{pmatrix} 
= 
\frac{1}{\det(M)}\Bigl( P \begin{pmatrix}
m_2n_3-m_3n_2 \\ m_3n_1-m_1n_3 \\ m_1n_2-m_2n_1
\end{pmatrix}
-a \begin{pmatrix}
n_2-n_3 \\ n_3-n_1 \\ n_1-n_2
\end{pmatrix}
- b\begin{pmatrix}
m_3-m_2 \\ m_1-m_3 \\ m_2-m_1
\end{pmatrix} \Bigr) \in \N^3
\end{gather*}
Since $\overrightarrow{\al}_P$ corresponds to the first term, i.e. $a=b=0$, and $\overrightarrow{\al}_P\not\in \N^3$ for all $P>0$
we have $\det(M)^{-1}(m_in_{i+1}-m_{i+1}n_i)<0$ for some $i\in\{1,2,3\}$ with convention $m_4=m_1$, $n_4=n_1$. Then for 
$P > \bigl( |a||n_i-n_{i+1}|+|b||m_{i+1}-m_i|\bigr)/|m_in_{i+1}-m_{i+1}n_i|$ the $i$-th coefficient is negative, so that the integral
in \eqref{eq:poweroff=3sumh} is zero. 
\end{proof}

\begin{rem}\label{rmk:support3} In case $\text{\rm rank}(M)=1$ the convex hull $C$ of $\bigl( (m_i,n_i)\bigr)_{i=1}^3$ equals $\{ (m,n)\}$, and in case $\text{\rm rank}(M)=3$ we have $(0,0)\in C$ if and only if $\exists \, 
\overrightarrow{\al}\in \Q^3_{\geq 0}$ with $M\overrightarrow{\al}= (1,0,0)^t$. From the proof of 
Proposition \ref{prop:MCfor3matrixelements} we see that $\int_{SU(2)}f(g)^P\, dg=0$ for all $P>0$ precisely when
$(0,0)\not\in C$ in the cases $\text{\rm rank}(M)\not=2$.
In case $\text{\rm rank}(M)=2$ the integral in \eqref{eq:poweroff=3sum} can have more than one non-zero term, and we 
have no control on possible cancellations. However, one expects that these cancellations cannot occur for all multiples of $P$ as well. The techniques
of Francoise et al. \cite{FranPYZ} might be useful in this regard considering it as polynomial identies in the $A_i$'s.
\end{rem}

\section{An alternative conjecture for the Mathieu Conjecture for $SU(2)$}\label{sec:MathieuConjforSU2}

Consider an arbitrary $SU(2)$-finite function $f=\sum_{i=1}^{k} A_i t^{\ell_i}_{m_i,n_i}$ with $A_i\neq 0$ for every $1\leq i \leq k$,
then applying the multinomial theorem shows that if 
\begin{equation}\label{eq:powergeneral}
\int_{SU(2)} (f(g))^P\, dg =
\sum_{\al_i\in\N, \sum_{i=1}^k \al_i=P} \binom{P}{\al_1,\dots,\al_k} \prod_{i=1}^k A_i^{\al_i} 
\int_{SU(2)} \prod_{i=1}^k \left(t_{m_i,n_i}^{\ell_i}\right)^{\al_i}(g) \, dg 
\not=0 
\end{equation}
for some $P>0$, then for some $(\al_1,\dots, \al_k)$ we have $\sum_{i=1}^k \frac{\al_i}{P} m_i =  \sum_{i=1}^k \frac{\al_i}{P} n_i=0$ by \eqref{eq:integralproducttlnmisnon0}, so 
$(0,0)$ is in the convex hull $C$ of $\bigl( (m_i,n_i)\bigr)_{i=1}^k$ over $\Q$.

\begin{conj}\label{conj:nonzerointegralpowers}
For any $SU(2)$-finite function 
$f=\sum_{i=1}^{k} A_i t^{\ell_i}_{m_i,n_i}$, $A_i\neq 0$ for all $1\leq i \leq k$, we have that
$\int_{SU(2)} (f(g))^P\, dg=0$ for all $P\in\N_{>0}$ if and only if 
$(0,0)$ is not contained in the closed convex hull of $\bigl( (m_i,n_i)\bigr)_{i=1}^k$.
\end{conj}

Lemma \ref{lem:1matrixeltpowerP}, Remarks \ref{rmk:supportfor2matrixelt}, \ref{rmk:support3} support
Conjecture \ref{conj:nonzerointegralpowers}. 

\begin{thm}\label{thm:alternativMC} Assume Conjecture \ref{conj:nonzerointegralpowers} holds, 
then the Mathieu Conjecture \ref{conj:mathieu} for $SU(2)$ holds. 
\end{thm}

\begin{proof} It suffices to show that $\int_{SU(2)} (f(g))^Pt^\ell_{a,b}(g)\, dg=0$ for $P$ sufficiently large 
assuming that $(0,0)$ is not contained in the closed convex hull $C$ of $\bigl( (m_i,n_i)\bigr)_{i=1}^k$. 
Using \eqref{eq:powergeneral} we see that $\int_{SU(2)} (f(g))^Pt^\ell_{a,b}(g)\, dg$ can only be
non-zero if $(-\frac{a}{P},-\frac{b}{P})\in C$. Since $(0,0)\not\in C$, we see that for $P$ sufficiently large 
this is not the case and the integral is zero. 
\end{proof}

\subsection*{Acknowledgement.} TD thanks Wilberd van der Kallen, Hjalmar Rosengren, Harm Derksen, Pablo Rom\'an  for useful discussions.
We thank Harm Derksen for pointing out Boyarchenko's result and Arno van den Essen for pointing out \cite{FranPYZ}.

\end{document}